\documentclass[11pt,a4paper,reqno]{amsart}

\usepackage{amsmath}
\usepackage{amsthm}
\usepackage{amsfonts}
\usepackage{amssymb}

\newcommand\R{{\mathbf{R}}}
\newcommand\A{{\mathcal{A}}}
\renewcommand\P{{\mathcal{P}}}
\newcommand\C{{\mathbf{C}}}
\newcommand\Z{{\mathbf{Z}}}
\newcommand\N{{\mathbf{N}}}

\newcommand\ZN{\mathbf{Z}_{N}}

\newcommand\E{\mathbf{E}}
\newcommand\T{\mathbf{T}}

\theoremstyle{plain}
  \newtheorem{theorem}{Theorem}

  \newtheorem{proposition}{Proposition}
  
  \newtheorem{lemma}{Lemma}
  
  \newtheorem{corollary}{Corollary}

\theoremstyle{remark}
  \newtheorem{remark}[subsection]{Remark}
  \newtheorem{remarks}[subsection]{Remarks}

\theoremstyle{definition}

\include{psfig}

\begin{document}

\title{Intersective polynomials and the primes}

\author{Th\'ai Ho\`ang L\^e}
\address{UCLA Department of Mathematics, Los Angeles, CA 90095-1596.}
\email{leth@math.ucla.edu}

\begin{abstract}
Intersective polynomials are polynomials in $\Z[x]$ having roots
every modulus. For example, $P_1(n)=n^2$ and $P_2(n)=n^2-1$ are
intersective polynomials, but $P_3(n)=n^2+1$ is not. The purpose of this note is to deduce, using results of Green-Tao \cite{gt-chen} and Lucier \cite{lucier}, that for any intersective polynomial $h$, inside any subset of positive relative density of the primes, we can find distinct primes $p_1, p_2$ such that
$p_1-p_2=h(n)$ for some integer $n$. Such a conclusion also holds in the Chen primes (where by a Chen prime we mean a prime number $p$ such that $p+2$ is the product of at most 2 primes).
\end{abstract}
\maketitle

\section{Introduction}
In the late 1970s, S\'{a}rk\"{o}zy and Furstenberg independently
proved the following, which had previously been conjectured by
Lovasz:

\begin{theorem}[S\'{a}rk\"{o}zy \cite{sarkozy1}, Furstenberg \cite{f1}, \cite{f2}]
If  $A$ is a subset of positive upper density of $\Z$, then there
are two distinct elements of $A$ whose difference is a perfect
square.
\end{theorem}

While Furstenberg used ergodic theory, S\'{a}rk\"{o}zy actually proved the
following finitary, quantitative form:

\begin{theorem}[S\'{a}rk\"{o}zy] \label{t1} Let $\delta>0$. Then provided $N$ is sufficiently large depending
on $\delta$, $N>N_0(\delta)$, any subset $A$ of $\{1, \ldots, N\}$
of size $\delta N$ contains two distinct elements $a, a' \in A$ such
that $a-a'$ is a perfect square.
\end{theorem}

We have the same conclusion if the set of the squares is replaced by
$\{p+1: p$ prime$\}$ or $\{p-1:p$ prime$\}$. More generally, we say
that a set $H \subset \Z^{+}$ is intersective if $H \cap(A-A) \neq
\emptyset$ for any set $A$ of positive upper density. We say that a
polynomial $h \in \Z[x]$ is intersective if the set $\{h(n): n \in \Z \} \cap (0,
\infty)$ is intersective. Thus S\'{a}rk\"{o}zy's theorem says that the
polynomial $h(n)=n^2$ is intersective.

Kamae and Mend\`{e}s France \cite{km} proved a criterion about
intersective sets. This gives a necessary and sufficient condition
for a polynomial to be intersective:
\begin{theorem}[Kamae-Mend\`{e}s France]
A polynomial $h \in \Z[x]$ is intersective if and only if for every
$d>0$, there exists $n$ such that $P(n) \equiv 0 \pmod{d}$.
\end{theorem}

For example, the polynomials $x^2$ and $x^2-1$ are intersective,
while $x^2+1$ is not (think of obstruction modulo 3). A polynomial having an integer root is
certainly intersective, but there are intersective polynomials which
do not have an integer root, e.g. the polynomials $(x^3-19)(x^2+x+1)$, or $(x^2-2)(x^2-3)(x^2-6)$. Berend and Bilu gave in \cite{bb} a procedure to determine whether or not a given polynomial
is intersective.

If $h$ is an intersective polynomial, denote by $D(h,N)$ the maximal size of a subset $A$ of $\{1,\ldots,N \}$ such that we cannot find distinct elements $a,a' \in A$ such that $a-a'=h(n)$ for some integer $n$. Thus necessarily $D(h,N)=o(N)$. It should be mentioned that like Furstenberg's method, Kamae and
Mend\`{e}s France's is qualitative, i.e., does not give any bound on
$D(h,N)$. In the case where $h(n)=n^2$, and more generally $h(n)=n^{k}$, the best bound is due to Pintz-Steiger-Szemer\'{e}di \cite{pss} and Balog-Pelikan-Pintz-Szemer\'{e}di \cite{bpps}. They proved that
$$D(n^{k},N) \ll_{k} N(\log N)^{-(1/4)\log\log\log\log N}$$
for $N$ sufficiently large depending on $k$. Note that this density already includes the primes. For general intersective polynomials, such a quantitative bound was obtained recently by Lucier \cite{lucier}. He proved that, for any intersective polynomial $h$ of degree $k$,
$$D(h,N)\ll_{h} N \frac{(\log \log N)^{\mu/(k-1)}}{(\log N)^{1/(k-1)}}$$
for $N$ sufficiently large depending on $h$, where $\mu=
\left\{
                                                                   \begin{array}{ll}
                                                                     3, & \hbox{if  $k=2$;} \\
                                                                     2, & \hbox{if  $k \geq 3$.}
                                                                   \end{array}
                                                                 \right.$

This density is weaker and does not include the primes. It may well be the case that the correct density includes the primes. However, we don't seek to improve upon Lucier's result, but rather use it, coupled with a ``transference principle'' to deduce a corresponding result for the primes.

Let $\P$ be a subset of $N$. For any subset $\A \subset
\P$, define the upper relative density of $\A$ with respect to $\P$ by
$\overline{d}_{\P}(\A)=\lim_{N \rightarrow
\infty}\frac{\sharp \{n\leq N: n \in \A \} }{\sharp \{n\leq N: n \in
\P \} }$. We will obtain the following:

\begin{theorem} \label{maint1}
For any intersective polynomial $h$, for any subset $\A$ of positive upper
relative density of the primes, there exist distinct elements $p_1,
p_2$ of $\A$ such that $p_1-p_2=h(n)$.
\end{theorem}

\begin{remarks} If $h(0)=0$, then this is a very special case of the result of Tao-Ziegler \cite{tao-ziegler}, which says that configurations $a+P_1(d),\ldots, a+P_{k}(d), d \neq 0$ exist in dense subsets of the primes, where $P_{i} \in \Z[x], P_{i}(0)=0$. Their starting point is a uniform version of the Bergelson-Leibman theorem, which says that such configurations exist in dense subsets of the integers. Tao-Ziegler's proof of the uniform version uses a lifting to a multidimensional version of the Bergelson-Leibman theorem and relies on the very fact that each $P_{i}(0)=0$. Therefore, it is not applicable to general intersective polynomials.
\end{remarks}

Following Green and Tao, let us call a prime $p$ a Chen prime if $p+2$ is either a prime or a product $p_1p_2$ of primes with $p_1, p_2>p^{3/11}$. The following result is due to Chen \cite{chen}:

\begin{theorem}[Chen]
Let $N$ be a large integer. The the number of Chen primes in the interval $[1, N]$ is at least $c_{1}N/\log^{2}N$ for some absolute constant $c_1>0$.
\end{theorem}

For a proof of Chen's theorem, see \cite{iwaniec}. Using this result as a ``black box'' we can show that the same conclusion holds for dense subsets of the Chen primes:

\begin{theorem} \label{maint2}
For any intersective polynomial $h$, for any subset $\A$ of positive upper
relative density of the Chen primes, there exist distinct elements
$p_1, p_2$ of $\A$ such that $p_1-p_2=h(n)$.
\end{theorem}

The idea of transferring results on dense subsets of the integers
to the primes originates with Green \cite{green}, in which he proved
an analog of Roth's theorem for the primes. Later on, other
transference principles have been devised by Green and Tao in
\cite{gt-chen} in which they proved the analog of Roth's theorem in
the Chen primes, and in \cite{gt-primes} in which they proved that
the primes contains arbitrarily long arithmetic progressions. These machineries have been used in a number of settings, such as random sets (\cite{tao-vu}, \cite{lh}) or the ring of polynomials over a finite field (\cite{ls}). We opt for the transference principle in \cite{gt-chen} since it is relatively simpler and more general than that in \cite{green}.
In a similar spirit, Li and Pan \cite{lp} proved that if $Q$ is a polynomial in $\Z[x]$ such that $Q(1)=0$, then inside any dense subset of the primes, we can find two distinct elements whose difference is of the form $Q(p)$ where $p$ is a prime number. It would be interesting to determine the class of all the polynomials $Q$ such that the same conclusion holds (other than those vanishing at 1).

\textbf{Acknowlegdments.} I would like to thank my advisor Terence Tao for helpful discussions during the preparation of this paper. I would like also to thank Craig Spencer for pointing me to Lucier's paper \cite{lucier} and for helpful comments on this paper. Part of this research was done when I was visiting the Mathematical Sciences Research Institute, Berkeley, and I am grateful to their hospitality.

\section{Notation and Preliminaries}
For two quantities $A,B$, we write $A=O(B)$, or $A\ll B$, or $B \gg
A$ if there is an absolute positive constant $C$ such that $|A| \leq CB$. If
$A$ and $B$ are functions of the same variable $x$, we write $A=o_{x
\rightarrow \infty}(B)$ if $A/B$ tends to 0 as $x$ tends to
infinity. If the constant $C$, (respectively, the rate of
convergence of $A/B$) depends on a parameter, e.g. $m$, then we
write $A=O_{m}(B)$ (respectively, $A=o_{m}(B)$). Quantities denoted
by the $C,c$ will stand for constants, which may change from line to
line. We denote by $\ZN$ the cyclic group on $N$ elements. This is
not to be confused with the ring of $p$-adic integers, which we also
denote by $\Z_{p}$, since we will make use of the latter very briefly
(in the introduction of auxiliary polynomials).

\subsection{Fourier analysis on $\ZN$} We will work primarily in a group $\ZN$ where $N$ is a large number. For a function $f:\ZN \rightarrow \C$ let us define its Fourier transform
by $\widehat{f}(\xi)=\E_{x \in \ZN} f(x) e_{N}(x \xi)$, where
$e_{N}(t)=e^{\frac{2 \pi i t}{N} }$, and $\E$ is the expectation. If
$f,g:\ZN \rightarrow \C$ are two functions, then $f*g$, the
convolution of $f$ and $g$, is defined by $f*g(x)=\E_{y \in
\ZN}f(y)g(x-y)$.
We also define the $l^{p}$-norm of $f$ to be $\|f\|_{p}= \left( \sum_{\xi \in \ZN} |f(\xi)|^{p} \right)^{1/p}$. We will often use a subset of $\ZN$ to denote its characteristic function.

We recall the basic properties of the Fourier transform:
\begin{itemize}
\item (Fourier inversion formula) $f(x)=\sum_{\xi \in \ZN} \widehat{f}(x) e_{N}(-x \xi)$
\item (Plancherel) $\sum_{\xi \in \ZN} \widehat{f}(\xi) \overline{\widehat{g}(\xi)}=\E_{x \in \ZN} f(x) \overline{g(x)}$
\item (Parseval) $\|\widehat{f}\|_{2}^2=\sum_{\xi \in \ZN}|\widehat{f}(\xi)|^2=\E_{x \in
\ZN}|f(x)|^2$
\item (Fourier transform of a convolution) $\widehat{f*g}(\xi)=\widehat{f}(\xi)\widehat{g}(\xi) \textrm{ for every } \xi
\in \ZN$
\end{itemize}

\subsection{Intersective polynomials}Let $h(x)=a_{k}x^{k}+\cdots+a_{0}$ be a fixed intersective polynomial of degree $k \geq 2$ throughout the paper. By a change of variables if need be, we may assume that $h$ and $h'$ are positive and increasing for $x \geq 0$.

If $f(x)=b_{k}x^{k}+\cdots+b_{0}$, let us denote by $b(f)=b_{k}$ and $B(f)=\frac{2}{|b_{k}|}(|b_{k-1}|+\cdots+|b_0|)$. Then if $b(f)>0$, we have $B(f') \leq B(f)$ and
\begin{equation}\label{b}
\frac{1}{2}b(f)x^{k} \leq f(x) \leq \frac{3}{2}b(f)x^{k}
\end{equation}
for $x \geq B(f)$ (\cite[Lemma 3]{lucier}).

If $f$ has integer coefficients, let us denote by $c(f)=\gcd(b_{k},\ldots,b_1)$, the content of $f$.

Suppose $f=a(x-\eta_1)^{e_1}\cdots(x-\eta_{r})^{e_{r}}$ in some splitting field. Let us denote by $\Delta(f)=a^{2k-2}\prod_{i \neq j}(\eta_{i}-\eta_{j})^{e_{i}e_{j}}$, the semidiscriminant of $f$. The semidiscriminant was first introduced by Chudnovsky \cite{chudnovsky}. When $f$ is separable then the semidiscriminant is simly the discriminant. It can be shown that $\Delta(f)$ is always a non-zero integer when $f \in \Z[x]$.

In order for the transference principle to work, we need not only one solution to $a-a'=h(n)$, but ``many'' (i.e., of the ``right'' order) of them. This is already established by Lucier. Another issue is that we will not be working directly with the primes, but rather affine images of primes (in congruences classes modulo $W$, where $W$ is a product of small primes meant to absorb obstruction at these primes). This technique is called the ``$W$-trick'' and is quite common in situations in arithmetic combinatorics where we want to transfer results on dense subsets of the integers to the primes \cite{green}, \cite{gt-chen}, \cite{gt-primes}, \cite{tao-ziegler}.

Thus instead of a single polynomial $h$, we will work with a family of polynomials $h_{W}$ parametrized by $W$, which are compositions of $h$ with affine maps. Our bounds need to be independent of $W$. As mentioned earlier, Tao-Ziegler's proof of the uniform version of the Bergelson-Leibman theorem does not apply to general intersective polynomials. Fortunately, the auxiliary polynomials introduced by Lucier serve well our purposes.

Note that the condition that $h$ has roots every modulo is equivalent to saying that $h$ has a root in $\Z_{p}$ for every prime $p$, where $\Z_{p}$ is the ring of $p$-adic integers. For each $p$ let us fix a root $z_{p} \in \Z_{p}$ of $h$. If $m$ is the multiplicity of $z_{p}$ as a root of $h$ then we define $\lambda(p)=p^{m}$. We can then extend $\lambda$ to a completely multiplicative arithmetic function on $\N$. It is easy to see that for every $d$, $d | \lambda(d) |d^{k}$.

Suppose $d=p_1^{\alpha_1}\cdots p_{s}^{\alpha_s}$ is the prime factorization of $d$. By the Chinese remainder theorem, let $r_{d}$ be the unique integer satisfying $-d < r_d \leq 0$ and $r_{d} \equiv z_{p} \pmod{ p_{i}^{\alpha_{i}}\Z_{p_{i}}}$ for every $i=1, \ldots, s$.

For any positive integer $d$ we define the polynomial $h_{d}(x)=\frac{h(r_{d}+dx)}{\lambda(d)}$. The properties of $h_{d}$, proved in \cite{lucier}, are summarized in the following lemma:

\begin{lemma} \label{aux}
\begin{enumerate}
    \item For every $d$, $h_{d}$ is a polynomial with integer coefficients and degree $k$. Furthermore, $h_{d}$ is also intersective.
    \item The polynomials $h(d), h'(d), h''(d)$ are positive and increasing for $x \ge 1$.
    \item For every $d, q>0$ then $(h_{d})_{q}=h_{dq}$.
    \item $b(h_{d}) \leq b(h_{d}) \leq d^{k-1}b(h)$.
    \item $B(h_{d}) \leq 2^{k-1}k(B(h)+2)$.
    \item $c(h_{d}) \leq | \Delta(h) |^{\frac{k-1}{2}} c(h)$, where $\Delta(h)$ is the semidiscriminant of $h$.
\end{enumerate}
\end{lemma}
\begin{remark}The last property is by far the most important, since our bounds on exponential sums involving $h_{d}$ will depend on $c(h_{d})$. The last two properties ensure that $B(h_{d})$ and $c(h_{d})$ can be bounded uniformly, no matter what $d$ is. The only quantity that can grow is $b(h_{d})$. We will see that this quantity is also within control if we keep $d$ smaller than a small power of $N$.
\end{remark}

\section{A uniform version of Lucier's theorem}
Let us first recall Lucier's main result in \cite{lucier}. Let $\delta>0$ and $A$ be a subset of $\{1, \ldots, N \}$ such that $|A|=\delta N$. For every $n$ let $r(h,n,A)$ be the number of couples $(a, a')$ of elements in $A$ such that $a-a'=h(n)$. Let $R(A,h)=\sum_{n \geq 0} h'(n)r(h,n,A)$.

\begin{theorem}[Theorem 5, \cite{lucier}] \label{lucier}
There is a constant $C(h,\delta)$ depending on $h$ and $\delta$ alone such that whenever $N$ is sufficiently large in terms of $h$ and $\delta$, the following estimate holds:
$$R(A,h) \geq c(h, \delta) |A|^2$$
\end{theorem}
Actually Lucier obtained the following estimate for $c(h, \delta)$:
$$c(h,\delta)= \exp \left( -c_1 \delta^{-(k-1)}\log^{\mu} \left( \frac{2}{\delta} \right) \right)$$
which is valid for $\delta \geq c_2 \frac{(\log \log N)^{\mu/(k-1)}}{\log N^{1/(k-1)}}$
where $c_1, c_2$ are constants depending on $h$ alone, and $\mu=
\left\{
                                                                   \begin{array}{ll}
                                                                     3, & \hbox{if  $k=2$;} \\
                                                                     2, & \hbox{if  $k \geq 3$.}
                                                                   \end{array}
                                                                 \right.$
As mentioned earlier, we need to work with the family $(h_{W})$ rather than with $h$ alone. The following gives a uniform version of Theorem \ref{lucier}:

\begin{theorem} \label{ul}
There is a constant $\kappa_1=\kappa_1(k)$ depending on $k$ alone, and a constant $C(h,\delta)$ depending on $h$ and $\delta$ alone such that whenever $N$ is sufficiently large in terms of $h$, the following estimate holds:
$$R(A,h_{W}) \geq C(h, \delta) |A|^2$$
for every $W<N^{\kappa_1}$, where the constant $C(h, \delta)$ is the same as in Theorem \ref{lucier} (but the range of validity of $N$ may be slightly different).
\end{theorem}

\begin{proof} Only a minor modification of Lucier's proof is needed. Lucier used a density increment argument based on the following:
\begin{lemma}[Lemma 31, \cite{lucier}] \label{iteration} Let $\varrho=\varrho(k)$ be defined by
$$\varrho=
\left\{
                                                                   \begin{array}{ll}
                                                                     1/4, & \hbox{if  $k=2$;} \\
                                                                     1/(8k^2(\log k + 1.5 \log \log k + 4.2)), & \hbox{if  $k \geq 3$.}
                                                                   \end{array}
                                                                 \right. $$
Define the function
$$\theta(x)=
\left\{
                                                                   \begin{array}{ll}
                                                                     \frac{x}{2 \log (2x^{-1})}, & \hbox{if  $k=2$;} \\
                                                                     x^{k-1}, & \hbox{if  $k \geq 3$.}
                                                                   \end{array}
                                                                 \right. $$
Let $N$ be large in terms of $h$, and assume that $$d \leq N^{\rho/4k^2}$$ Let $A$ be a subset of $\{1,\ldots, N\}$ with size $\delta N$ such that $$\delta \geq N^{-\varrho/2k}$$ If $R(h_{d},A)\leq \frac{1}{64}|A|^2$, then there exist positive integers $d'$ and $N'$, and a set $A' \subset \{1, \ldots, N' \}$ such that the following holds:
\begin{itemize}
\item $W(h_{d'},A') \leq W(h_{d},A)$,
\item $\delta' \geq \delta(1+C_1 \theta(\delta))$,
\item $C_2 \delta^{2k^2} N \leq N' \leq N$,
\item $d \leq d' \leq C_3 \delta^{-k}d$.
\end{itemize}
where $C_1,C_2,C_3$ are positive constants that depend only on $h$.
\end{lemma}
Following Lucier, suppose that $$\delta \geq C \frac{(\log \log N)^{\mu/(k-1)}}{(\log N)^{1/(k-1)}}$$ for $C$ a constant chosen later, that depends on $h$ alone. Let $$Z = [8C_1^{-1} \delta^{-(k-1)}(\log 2\delta^{-1})^{\mu-1}]$$ Suppose, for a contradiction, that that $R(h,A)\leq \frac{1}{64}(C_2^2 \delta^{4k^2})^{Z}|A|^2$. Lucier constructed a sequence of quadruples $\{ (N_{i}, A_{i}, \delta_{i},d_{i}) \}_{i=0}^{Z}$, where $N_{i},d_{i}$ are positive integers, $A_{i} \subset \{1, \ldots, N \}$, $\delta_{i}=|A_{i}|/N_{i}$,  satisfying the properties:
\begin{itemize}
\item $(N_0, A_0,\delta_0,d_0)=(N,A,\delta,1)$
\item $R(h_{d_{i}}, A_{i}) \leq \frac{1}{64}(C_2^2 \delta^{4k^2})^{Z-i} |A_{i}|^2$
\item $\delta_{i} \geq \delta_{i-1}(1+C_1 \theta(\delta_{i-1}))$
\item $C_2 \delta_{i-1}^{2k^2}N_{i-1}\leq N_{i} \leq N_{i-1}$
\item $d_{i-1} \leq d_{i} \leq C_3 \delta_{i}^{-k}d_{i-1}$
\end{itemize}
where $C_1, C_2, C_3$ are constants as in Lemma \ref{iteration} above. We can perform the iteration at step $l$ as long as the conditions of Lemma \ref{iteration} is valid:
\begin{enumerate}
\item $N_{l}$ is large in terms of $h$. Indeed, if we choose $C$ large enough we can ensure that $N_{l}\geq N^{1/2}$ for all $0 \leq l \leq Z$.
\item \label{check} $d_{l} \leq N_{l}^{\rho/4k^2}$. Indeed, we have the inequality $\log d_{l} \ll_{h} C^{-1} \log N + \log d_{0}$, so if $C$ is large enough in terms of $h$ this is satisfied.
\item $\delta_{l} \geq N_{l}^{-\varrho/2k}$. This too is ensured if $C$ is large enough.
\end{enumerate}
A calculation shows that we will end up with $\delta_{Z}>1$, a contradiction.

Now if the initial values are $(N_{0}, A_{0}, \delta_{0},d_{0})=(N, A, \delta,W)$ (instead of $(N, A, \delta, 1)$) then the same iteration goes through. The only thing that needs to be checked is the condition (\ref{check}) above. But we can ensure this by choosing $C$ sufficiently large depending on $h$ alone, as long as we keep $W$ smaller than $N^{\kappa_1}$, for $\kappa_1=\frac{\varrho}{16k^2}$, say.
\end{proof}

\section{A transference principle for intersective polynomials} \label{transference}
\subsection{An exponential sum estimate}
\begin{lemma} \label{w}
Let $f \in \Z[x]$ be a polynomial of degree $k$, and assume that $f$ is positive and increasing for $x \leq 0$. Then there is an integer $s_0(k)$ depending on $k$ alone, such that whenever $s \geq s_0$, we have
$$\int_{\T} \left| \sum_{n=1}^{N} f'(n)e(\alpha f(n))\right|^{2s} \ll_{s} c(f) f(N)^{2s-1}$$
for $N \geq B(f)$.
\end{lemma}
\begin{remarks}
This is essentially \cite[Lemma 2.6]{lp}, where Li and Pan showed that we can take $s_{0}=k2^{k+1}$. This result is standard in the context of Waring's problem, so we will skip the proof. It may be possible to improve upon the value of $s_0$ using Vinogradov's method, but this is not important since all we need is the existence of such a number $s_{0}$. The condition $N \geq B(f)$ is needed in order to guarantee that $\sum_{n=1}^{N} f'(n) \ll f(N)$.
\end{remarks}

Let us denote $S_{f}(x)= S_{N,f}(x)=
\left\{
                                                                   \begin{array}{ll}
                                                                     f'(n), & \hbox{if $0<x<N/2$ and $x=f(n)$ for some $n \in \Z$;} \\
                                                                     0, & \hbox{otherwise.}
                                                                   \end{array}
                                                                 \right.$
and consider $S_{f}$ as a function on $\ZN$.
\begin{corollary}
For $s \geq s_{0}(k)$, and for $N \gg b(f)B(f)^{k}$, we have
$$\left\| \widehat{S_{f}}\right\|_{2s} \ll_{s} (c(f))^{1/2s} $$
\end{corollary}
\begin{proof}
Let $M$ be the largest integer such that $f(M) < \frac{N}{2}$. In view of (\ref{b}), if $b(f)B(f)^{k} \ll N$ then $M \geq B(f)$. We can therefore apply Lemma \ref{w} and have:
\begin{eqnarray}
\left\| \widehat{S_{N,f}}\right\|_{2s}^{2s} &=& \frac{1}{N^{2s}} \sum_{\xi \in \ZN} \left| \sum_{x \in \ZN} S_{f}(x)e_{N}(\xi x) \right|^{2s} \nonumber \\
&=& \frac{1}{N^{2s-1}} \sum_{\substack{ n_1,\ldots,n_s, m_1,\ldots, m_{s} \in \{1, \ldots, M\}\\ f(n_1)+\cdots+f(n_s) = f(n_1)+\cdots+f(n_s) \ }} f'(n_1)\cdots f'(n_s) f'(m_1)\cdots f'(m_s) \nonumber \\
&=& \frac{1}{N^{2s-1}} \int_{\T} \left| \sum_{n=1}^{M} f'(n)e(\alpha f(n))\right|^{2s} \nonumber \\
&\ll_{s}& \frac{1}{N^{2s-1}} c(f) f(M)^{2s-1} \nonumber \\
&\ll_{s}& c(f) \nonumber
\end{eqnarray}
\end{proof}
From this it immediately follows that
\begin{corollary} \label{c1}
There is a constant $\kappa_2=\kappa_2(k)$ such that for $s \geq s_0$, and for $N$ sufficiently large depending on $h$, we have
$$\left\| \widehat{S_{N,h_{W}}}\right\|_{2s} \ll_{s,h} 1 $$
for every $W<N^{\kappa_2}$.
\end{corollary}
\begin{proof}
Lemma \ref{aux} tells us that $c(h_{W})$ is uniformly bounded in terms of $h$. Thus we need $b(h_{W})B(h_{W})^{k} \ll N$ for all $W \leq N^{\kappa_2}$. But this also follows from Lemma \ref{aux}. Actually we may take $\kappa_2(k)=1/k$.
\end{proof}

\subsection{A transference principle}
Let us reformulate Theorem \ref{lucier} under the following form:
\begin{proposition} \label{s}
There is a constant a constant $c(h,\delta)$ such that the following holds. If $f : \ZN \rightarrow [0, \infty)$ is a function such that $\E_{\ZN} f \geq \delta$, then
$$\sum_{a \in \ZN} \sum_{d \in \ZN} f(a) f(a+n) S_{h}(d) \geq c(h,\delta) N^2$$
for $N$ sufficiently large depending on $h$ and $\delta$.
\end{proposition}

We are now in a position to state the following transference principle for intersective polynomials:

\begin{proposition}\label{trans} Let $\eta, \delta, M, q$ be positive parameters such that $2 <q < \frac{4s_0}{2s_0-1}$, where $s_0=s_0(k)$ as in Lemma \ref{w}. Suppose $f, \nu$ are function $\ZN \rightarrow \R$ satisfying the following conditions:
\begin{enumerate}
    \item \label{con1} $0 \leq f \leq \nu$
    \item \label{con2} $\E_{n \in \ZN} f(n) \geq \delta$
    \item \label{con3} $\nu$ satisfies the pseudorandom condition $|\widehat{\nu}(\xi)-1_{\xi=0}|\leq \eta$ for all $\xi \in \ZN$.
    \item \label{con4} $\|\widehat{f}\|_{q} \leq M$.
\end{enumerate}
Then for $N$ large enough depending on $h$ and $\delta$, we have $$\sum_{a \in \ZN} \sum_{d \in \ZN} f(a) f(a+d) S_{h}(d)
\geq \left( \frac{1}{2}c(h,\delta) - O_{M,q, \delta}(\eta) \right) N^2$$
\end{proposition}
We proceed as in \cite[Proposition 5.1]{gt-chen}. Let us recall in the form of a lemma the following decomposition result contained in the proof of \cite[Proposition 5.1]{gt-chen}:
\begin{lemma}Suppose $0<\epsilon <1$. Let
$$\Omega=\{a \in \ZN: |\widehat{f}(a)| \geq \epsilon \}$$ and
$$B=B(\Omega, \epsilon)=\{m \in \ZN: |1-e_{N}(am)| \geq \epsilon \textrm{ for all } a \in \Omega \}$$
Let $$f_1(n)=\E_{m_1, m_2 \in B}f(n+m_1-m_2)$$ and
$f_2=f-f_1$ is the uniform part. Then $f_1$ and $f_2$ satisfy the following properties:
\begin{enumerate}
    \item $0 \leq f_1 \leq 1 + (N/|B|)\eta$,
    \item $\E_{\ZN}(f_1)=\E_{\ZN}(f)$
    \item $\| \widehat{f_2}(\psi)\|_{\infty} \leq 3(1+\eta)\epsilon$,
    \item For every $\xi \in \ZN$, we have $|\widehat{f_1}(\xi)|,|\widehat{f_2}(\xi)| \leq |\widehat{f}(\xi)|$.
\end{enumerate}
\end{lemma}
\begin{proof}[Proof of Proposition \ref{trans}]
We write
\begin{eqnarray}
\sum_{a,d \in \ZN}  f(a) f(a+d) S_{h}(d) &=& \sum_{a, d \in \ZN} f_1(a)f_1(a+d) S_{h}(d) + \sum_{a, d \in \ZN} f_1(a) f_2(a+d) S_{h}(d) \nonumber \\
&+& \sum_{a, d \in \ZN} f_2(a) f_1(a+d) S_{h}(d) + \sum_{a, d \in \ZN}
f_2(a) f_2(a+d) S_{h}(d) \nonumber
\end{eqnarray}
Note that since $\| \widehat{f}(\xi) \|_{q} \leq M$, we have $| \Omega | \leq (M/ \epsilon)^{q}$. Also, $|B| \geq (\epsilon/C)^{|\Omega|}$ for some absolute constant $C$. Thus we have $0 \leq f_1 \leq 1+(C/\epsilon)^{(M/\epsilon)^{q}}\eta=1+O_{M, \epsilon, q}(\eta)$.
Applying Proposition \ref{s} to the function $f_1$ (possibly modified by $O_{M,q,\epsilon}(\eta)$), we have
$$\sum_{a \in \ZN} \sum_{d \in \ZN} f_1(a) f_1(a+d) S_{h}(d) \geq (c(h,\delta)-O_{M,q,\epsilon}(\eta))N^2$$

Our goal is to show that the three last terms are small in absolute value. We consider the second term; the other two terms are treated similarly. We have
\begin{eqnarray} \left| \sum_{a, d \in \ZN} f_1(a) f_2(a+d)
S_{h}(d) \right| &=& N \left| \sum_{a \in \ZN} f_1(a) f_2*S_{h}(a) \right| \nonumber \\
&=& N^2 \left| \sum_{\xi \in \ZN} \overline{\widehat{f_1}(\xi)} \widehat{f_2}(\xi) \widehat{S_{h}}(\xi) \right| \nonumber \\
&\leq& N^2 \sum_{\xi \in \ZN} |\widehat{f_1}(\xi)| |\widehat{f_2}(\xi)|
|\widehat{S_{h}}(\xi)| \nonumber
\end{eqnarray}
By H\"{o}lder's inequality, $$\sum_{\xi \in \ZN} |\widehat{f_1}(\xi)| |\widehat{f_2}(\xi)|
|\widehat{S_{h}}(\xi)| \leq \|\widehat{f_2}\|^{t}_{\infty} \|\widehat{f_1}\|_{q} \| \widehat{f_2}\|^{1-t}_{q}  \|\widehat{S_{h}} \|_{2s_0}$$ where $t>0$ is such that $\frac{2-t}{q}+\frac{1}{2s_0}=1$. By Corollary \ref{c1} we know that $\|\widehat{S_{h}} \|_{q} \ll_{q} 1 $. Thus $\sum_{a,d \in \ZN}  f(a) f(a+d) S_{h}(d) \ll_{q} (1+\eta)^{t}\epsilon^{t}M^{2-t}$. We have similar estimates for the other two terms. Thus by choosing $\epsilon$ sufficiently small depending on $M,q, \delta$, the contribution of the three last terms is less than $\frac{1}{2}c(h,\delta)$.

Therefore,
$$\sum_{a \in \ZN} \sum_{d \in \ZN} f(a) f(a+d) S(d) \geq \left( \frac{1}{2}c(P,\delta)-O_{M,q}(\eta) \right) N^2$$
as required.
\end{proof}
From Proposition \ref{trans} we immediately have the following:
\begin{corollary} \label{trans2}
Let $\kappa=\min(\kappa_1, \kappa_2)$ where $\kappa_1$ is the constant in Theorem \ref{ul} and $\kappa_2$ is the constant in Corollary \ref{c1}. Then under the same hypothesis as in Proposition \ref{trans}, we have
$$\sum_{a \in \ZN} \sum_{d \in \ZN} f(a) f(a+d) S_{h_{W}}(d)
\geq \left( \frac{1}{2}c(h,\delta) - O_{M,q, \delta}(\eta) \right) N^2$$
for all $N$ large enough depending on $h$ and $\delta$, and $W<N^{\kappa}$.
\end{corollary}

\section{Construction of a pseudorandom measure that majorizes the primes} \label{construction}
In this section we will find functions $f, \nu$ satisfying the
conditions of Proposition 1 such that $f$ is supported on the Chen
primes. This is done exactly the same way as in the proof of \cite[Theorem 1.2]{gt-chen}, the main tool being the Hardy-Littlewood majorant property for objects called ``enveloping sieves''.

Let us recall the settings from \cite{gt-chen}. Consider $F=\prod_{j=1}^{k}(a_{j}n+b_{j})$, a product of $k$ linear factors with integer coefficients, no two linear factors are rational multiples of each other.

Let $X=X(F)=\{n \in \Z^{+}: F(n) \textrm{ is the product of $k$ primes} \}$.
For any $q \geq 1$, let  $X_{q}=\{n \in \Z_{q}:(F(n), q)=1 \}$. Thus $X_{R!}=\{n \in \Z: (d,F(n))=1 \textrm{ for all }1 \leq d \leq R \}$.
Let $\gamma(q)=\frac{|X_{q}|}{q}$. We assume that $\gamma(q)>0$ for all $q \geq 1$. Let $\mathfrak{S}_{F}$ be the singular series $\mathfrak{S}_{F}=\prod_{p \textrm{ prime}}\frac{\gamma(p)}{\left(1-\frac{1}{p} \right)^{k}}$.

\begin{proposition}[Proposition 3.1, \cite{gt-chen}]
Let $F$ be as above, with coefficients $a_i,b_i$ satisfying $|a_i|,|b_i| \leq N$. Let $R \leq N$ be a large integer.
Then there is a non-negative function $\beta := \beta_R: \mathbb{Z} \rightarrow \R^+$, called the envelopping sieve associated to $F$ and $R$,
with the following properties:
\begin{itemize}
\item[(i)] \textup{(Majorant property)} We have
\begin{equation}\label{sing-series} \beta(n) \gg_k \mathfrak{S}_F^{-1} \log^k R \mathbf{1}_{X_{R!}}(n)
\end{equation}
for all integers $n$. In particular, $\beta(n)$ is non-negative.
\item[(ii)] \textup{(Crude upper bound)}  We have
\begin{equation}\label{upper-crude}
\beta(n) \ll_{k,\epsilon} N^{\epsilon}
\end{equation}
for all $0 < n \leq N$ and $\epsilon > 0$.
\item[(iii)] \textup{(Fourier expansion)} We have
\begin{equation}\label{eq3.20} \beta(n) \; = \; \sum_{q \leq R^2} \sum_{a \in \Z_q^*}w(a/q)e_q(-an),\end{equation}
where $w(a/q) = w_R(a/q)$ obeys the bound
\begin{equation}\label{star} |w(a/q)| \; \ll_{k,\epsilon} \; q^{\epsilon - 1}\end{equation}
for all $q \leq R^2$ and $a \in \Z^*_q$.  Also we have $w(0) = w(1) = 1$.
\item[(iv)] \textup{(Fourier vanishing properties)} Let $q \leq R^2$ and $a \in \Z_q^*$.  If $q$ is not square-free,
then $w(a/q) = 0$.  Similarly, if $\gamma(q) = 1$ and $q > 1$, then $w(a/q) = 0$.
\end{itemize}
\end{proposition}

It should be mentioned that all the implied constants depend on $k$, but not on $F$. Moreover, $\beta_{R}$ enjoys the following properties:
\begin{proposition}[Discrete majorant property, Proposition 4.2, \cite{gt-chen}]For every $q>2$, we have
$$\left( \sum_{b \in \ZN} \left| \E_{1 \leq n \leq N} a_{n} \beta_{R}(n)e_{N}(-bn) \right|^{q} \right)^{1/q}\ll_{q,k} \left(\E_{1\leq n \leq N}|a_{n}|^2 \beta_{R}(n) \right)^{1/2}$$
\end{proposition}

\begin{proposition}[Lemma 4.1, \cite{gt-chen}]
Suppose $R \leq \sqrt{N}$. Then $\E_{1\leq n \leq N}\beta_{R}(n)\ll 1$.
\end{proposition}

Suppose $\A$ is a subset of positive relative density of the primes. Let $t$ be a large number (independent of $N$), and $W=W_{t}=\prod_{p\leq t}p$. We will assume at all times that $W<N^{\kappa}$, where $\kappa$ is the constant as in Corollary \ref{trans2}. By the pigeonhole principle we can choose $b \in X_{W}$ such that the set $X=\{0 \leq n \leq N/2: \lambda(W)n+b \in \A \}$ satisfies
\begin{eqnarray}
|X| &\gg& \frac{1}{\phi(\lambda(W))}\frac{N \lambda(W)}{\log (N\lambda(W))} \nonumber \\
 &\gg& \frac{\lambda(W)}{\phi(\lambda(W))} \frac{N}{\log N} \nonumber \\
 &\gg& \prod_{p \leq t}(1-1/p)^{-1} \frac{N}{\log N} \nonumber \\
 &\gg& \log t \frac{N}{\log N} \label{eq1}
\end{eqnarray}
for infinitely many $N$. We may assume henceforth that $N$ satisfies the inequality (\ref{eq1}). Let us now consider the polynomial $F(n)=\lambda(W)n+b$. Then it is easy to see that $\mathfrak{S}=\prod_{p \leq t}(1-1/p)^{-1} \ll \log t$.

Now let $R = [N^{1/20}]$ and let $\beta_{R}:\Z \rightarrow R^{+}$ be the enveloping sieve associated to $F$ and $R$. Let $\nu$ be the restriction of $\beta$ on $\{1,\ldots,N \}$ which may be regarded as a function on $\ZN$. Then we have $\nu(n) \gg \mathfrak{S}^{-1} \log N 1_{X}(n) \gg \frac{1}{\log t} \log N 1_{X}(n)$.
\begin{lemma} [Lemma 6.1,\cite{gt-chen}] \label{pseudorandom}
$\widehat{\nu}(a)=\delta_{a,0} + O(t^{-1/2})$.
\end{lemma}

\begin{proof}[Proof of Theorem \ref{maint1}]
Let us now define the function $f: \ZN \rightarrow \R^{+}$ by  $$f(n)=c \frac{\log N}{\log t} 1_{X}(n)$$ Let us verify the conditions of Proposition \ref{trans}. Clearly $0 \leq f \leq \nu$ for $c$ appropriately chosen, and $\E_{\ZN}f \geq \delta >0$, where $\delta$ depends only on the upper relative density of $\A$ in the primes.

Fix any $2<q<4s_{0}/(2s_{0}-1)$. By Propositions 4 and 5 (for the sequence $a_{n}=\frac{f(n)}{\nu(n)}$, with the convention that $a_{n}=0$ if $f(n)=\nu(n)=0$), we have
$$\| \widehat{f} \|_{q}= \left( \sum_{b \in \ZN} \left| \E_{1 \leq n \leq N} f(n)e_{N}(-bn) \right|^{q} \right)^{1/q} \ll \left( \E_{1\leq n \leq N}\frac{f(n)^2}{\nu(n)} \right)^{1/2} \ll \left(\E_{1\leq n \leq N}\nu(n) \right)^{1/2} \ll 1$$
Thus the condition (\ref{con4}) of Proposition \ref{trans} is satisfied.
Finally, the condition (\ref{con3}) of Proposition \ref{trans} follows from Lemma \ref{pseudorandom} with $\eta=O(t^{-1/2})$.

Proposition \ref{trans} now tells us that
\begin{equation}
\sum_{a,d \in \ZN}f(a)f(a+d)S_{h_{W}}(d) \geq c(h,\delta) - O(t^{-1/2}) \label{eq2}
\end{equation}
for some constant $c$ depending on $h$ and $\delta$, for $N$ sufficiently large depending on $h$, and for every $W \leq N^{\kappa}$. Thus for $t$ sufficiently large depending on $h$ and $\delta$, for $N$ sufficiently large depending on $t$, we have $\sum_{a,d \in \ZN}f(a)f(a+d)S_{h_{W}}(d)>0$, which implies the existence of a couple $a, a' \in X$ and $d$ such that $$a-a'=h_{W}(d)=\frac{h(Wd+r_{W})}{\lambda(W)}\neq 0$$
A priori, this is an equality in $\ZN$, but since $a, a', h_{W}(d) <\frac{N}{2}$, this is an equality in $\Z$. Therefore, $h(Wd+r_{W})=(\lambda(W) a +b)-(\lambda(W) a' +b)$ is the difference of two elements of $\A$, as desired.
\end{proof}

\begin{proof}[Proof of Theorem \ref{maint2}] The proof goes along the lines that of Theorem \ref{maint1}. Suppose $\A$ is a subset of positive relative density of the Chen primes. This time, we consider $X=\{0 \leq n \leq N/2: \lambda(W)n+b \in \A \}$ for some appropriately chosen $b$, $F=(\lambda(W)n+b)(\lambda(W)n+b+2))$, and $f=c \frac{\log N}{\log^2 t}1_{X}(n)$.
\end{proof}
\begin{remarks}What we have proved so far is that not only is there a couple $p_1, p_2$ such that $p_1-p_2=h(n)$ for some $n$, but the number of such couples is of the correct magnitude. More precisely, if $\A$ is a subset of positive upper relative density of the primes, then we have
$$\sharp \{(p_1,p_2): p_1, p_2 \in \A, p_1, p_2 \leq N, p_1-p_2=h(n) \textrm{ for some }n \} \gg \frac{N^{1+1/k}}{\log^2 N}$$
where the implied constant depends only on $h$ and the upper relative density of $\A$. A similar conclusion holds for subsets of positive relative density of the Chen primes.
\end{remarks}

\section{Further discussions}
\subsection{A word on bounds}
Recall that in the estimate (\ref{eq2}), $c(h, \delta)$ has the form $$c(h,\delta)= \exp \left( -c_1 \delta^{-(k-1)}\log^{\mu} \left( \frac{2}{\delta} \right) \right)$$
while the error term $O(t^{-1/2})$ takes the form $(C/ \epsilon)^{(M/\epsilon)^{q}}t^{-1/2}$, where $M, C$ are constants depending at most on $k$, $c_1$ is a constant depending on $h$, and $\epsilon$ is a power of $c(h, \delta)$. Recall that $t \ll \log W \ll_{k} \log N$. A calculation shows that the error term is dominated by the main term as long as
$$\delta \gg_{h} \frac{(\log_4 N)^{\mu/(k-1)}}{(\log_3 N)^{1/(k-1)}}$$
(where $\log_{i}$ denotes the number of times the $\log$ has to be taken). Thus we have proved that, inside any subset of size $\gg_{h} \frac{N}{\log N} \frac{(\log_4 N)^{\mu/(k-1)}}{(\log_3 N)^{1/(k-1)}}$ of the primes in $\{1, \ldots, N\}$, there must exist two distinct elements $p_1, p_2$ such that $p_1-p_2=h(n)$ for some $n \in \Z$. A similar conclution holds for the Chen primes. Such a bound is of course far weaker than Pintz-Steiger-Szemer\'{e}di type bounds.
\subsection{On the transference principle}
Our transference principle relies on two properties of the intersective set $H= \{ h(n): n \in \Z \}$, namely Theorem \ref{lucier} and Proposition \ref{w}. Theorem \ref{lucier} says that the number of solutions to $a-a'=m$ where $a, a'$ are in any given dense set and $m \in H$ is of the expected order of magnitude. Proposition \ref{w} requires that the number of representations of any number as a sum of elements of $H$ be bounded by the expected order of magnitude. We may ask for which other classes of intersective sets these two properties hold. A natural candidate is the set of values of polynomials of prime variables. It is known that the set $\{Q(p): p \textrm{ prime}\}$ is intersective, where $Q \in \Z[x]$ is such that $Q(1)=0$; however there are other examples such as $Q(p)=(p-3)(p-5)$.
Other examples of intersective sets include $\{ [\alpha n^2]: n \in \Z^{+} \}$ for irrational $\alpha$, and more generally the set of values of certain generalized polynomials (whose intersectivity is established in \cite{bh}). We may ask the same question for generalized polynomials in prime variables such as $\{[\alpha p^2]: p \textrm{ prime}\}$ for $\alpha$ irrational (whose intersectivity is not yet established yet but very plausible). However, as we have seen how the $W$-trick comes into play, we will have to take into account uniform versions of the two properties, which don't seem to be a simple matter.

\end{document}